%% file: cocartmodelstructure.tex
\documentclass{scrartcl}

\KOMAoptions{headings=small}
\usepackage[usenames, dvipsnames]{xcolor}

\usepackage{tikz}
\usepackage{tikz-cd}
\tikzcdset{
 arrow style=tikz,
 diagrams={>={Straight Barb[scale=0.715]}}
 }
\usepackage{braket}
\usepackage{graphicx}
\usepackage{amsmath}
\usepackage{amsthm}
\usepackage{csquotes}

\usepackage[T1]{fontenc}
\usepackage[bitstream-charter]{mathdesign}
\usepackage[colorlinks=true, linkcolor=NavyBlue, citecolor = OliveGreen]{hyperref}
\usepackage{microtype}

\theoremstyle{definition}
\newtheorem{sub}{}[section]
\newtheorem{defn}[sub]{Definition}

\newtheorem{Remark}[sub]{Remark}
\newtheorem{Example}[sub]{Example}

\theoremstyle{plain}
\newtheorem{thm}[sub]{Theorem}
\newtheorem{lem}[sub]{Lemma}

\newtheorem{coro}[sub]{Corollary}
\newtheorem{prop}[sub]{Proposition}

\newcommand{\mcyl}{(\Delta^1)^\sharp}

\newcommand{\Fun}{\mathrm{Fun}}
\newcommand{\mSet}{\mathbf{sSet}^+}

\newcommand{\Hom}{\mathrm{Hom}}
\newcommand{\Map}{\mathrm{Map}}
\newcommand{\coCart}{\mathbf{coCart}}

\newcommand{\innerhorn}{\Lambda^n_k}
\newcommand{\nsimplex}{\Delta^n}

\newcommand{\coslice}{(a/A)^\sharp}

\title{A note on coCartesian fibrations}
\author{Hoang Kim Nguyen}
\date{}

\begin{document}

\begin{titlepage}
\maketitle
\begin{abstract}
We prove properness of (co)Cartesian fibrations as well as a straightening and unstraightening equivalence, which is compatible with cartesian products, when the base is the nerve of a small category.
\end{abstract}

\tableofcontents

\end{titlepage}

\section{Introduction}

\begin{sub}
An essential tool in higher category theory is the theory (co)Cartesian fibrations; the higher categorical analogue of Grothendieck (op)fibrations. By Lurie's straightening equivalence, coCartesian fibrations correspond to functors to the \(\infty\)-category of \(\infty\)-categories and manipulating coCartesian fibrations is often the preferred way to construct such functors.
\end{sub}

\begin{sub}
The first proof of the straightening equivalence was given by Lurie in \cite{lurie} and relies on a comparison with simplicial categories. Lurie then uses the straightening equivalence to prove important fundamental properties of coCartesian fibrations. An efficient streamlined proof has appeared in \cite{short_StrUnStr}. This proof also uses the comparison with simplicial categories but proves the fundamental properties of coCartesian fibrations first and derives the equivalence from there.
\end{sub}

\begin{sub}
In upcoming joint work with Denis-Charles Cisinski we provide a new construction of the \(\infty\)-category of \(\infty\)-categories inspired by the construction of universes in semantic models of type theory. The straightening equivalence is then derived from a property of the universe of \(\infty\)-categories: \emph{directed univalence}. 
\end{sub}

\begin{sub}
The present article serves as backbone to the above mentioned work on directed univalence. Similar to \cite{short_StrUnStr}, we will prove fundamental properties of coCartesian fibrations from scratch using only methods from (marked) simplicial sets. As such, there is some overlap with \cite{short_StrUnStr} and we indicate whenever this is the case. Since the purpose of this article is to lay the technical foundations to prove a straightening theorem, we are careful to prove everything we need without using a straightening theorem.
\end{sub}

\begin{sub}
The main contributions of this article are the following:
\begin{itemize}
	\item A proof that pullback along (co)Cartesian fibrations preserves Joyal equivalences (Theorem \ref{joyalproperness}). This has been proven by Lurie \cite[Corollary 3.3.1.2]{lurie} using the full power of his straightening equivalence. Although we don't need this result, a weaker assertion which is much easier to prove suffices for us (Proposition \ref{ezproper}), we are not aware of a proof of this result using only elementary methods, so we include here.
	\item A proof of the straightening equivalence for nerves of categories (Theorem \ref{rectificationeq}). We can't completely avoid straightening, but this version avoids simplicial categories. This is not the full straightening equivalence as it only applies when the base is the nerve of a small category. The upshot however is that the proof is relatively short and we also prove a compatibility with cartesian products. 
\end{itemize}
\end{sub}

\paragraph{Acknowledgements} The author would like to thank Denis-Charles Cisinski for many useful discussions and for his encouragement. Most of this work has been completed when the author was a member of the \emph{SFB 1085 Higher Invariants} funded by the \emph{Deutsche Forschungsgesellschaft (DFG)}.

\section{Reminder on the coCartesian model structure}

\begin{sub}
We denote $\mathbf{sSet}^+$ the category of marked simplicial sets. It's objects are given by pairs $(K,E_K)$ where $K$ is a simplicial set and $E_K$ is a set of 1-simplices of $K$ containing all the degenerate 1-simplices, called \emph{marked edges}. It's morphisms are maps of simplicial sets preserving the marked edges. In general we will denote a marked simplicial set by \(K^+\).
\end{sub}

\begin{sub}
The forgetful functor
\[
\mathbf{sSet}^+ \to \mathbf{sSet}
\]
has both a left and a right adjoint. We denote the left adjoint by 
\[
(\cdot)^\flat \colon \mathbf{sSet} \to \mathbf{sSet}^+
\]
Given a simplicial set $A$, the marked simplicial set $A^\flat$ has precisely the degenerate 1-simplices marked. The right adjoint will be denoted by
\[
(\cdot)^\sharp \colon \mathbf{sSet}\to \mathbf{sSet}^+
\]
Given a simplicial set $B$, the marked simplicial set $B^\sharp$ has all 1-simplices marked. 
\end{sub}

\begin{sub}
The functor $(\cdot)^\sharp$ has a further right adjoint, denoted by 
\[
\mu \colon \mathbf{sSet}^+\to \mathbf{sSet}
\]
Given a marked simplicial set $A^+$, the simplicial set $\mu (A^+)$ is the simplicial subset of $A$ spanned by the marked edges.
\end{sub}

\begin{defn}\label{markedgenerators}
We define the class of \emph{marked left anodyne extensions} to be the smallest saturated class containing the morphisms
\begin{itemize}
	\item[(A1)] $(\Lambda^n_k)^\flat \to (\Delta^n)^\flat$ for $n \geq 2$ and $0<k<n$,
	\item[(A2)] $J^\flat \to J^\sharp$, 
	\item[(B1)] $\mcyl \times (\Delta^1)^\flat \cup \{0\} \times \mcyl \to \mcyl \times \mcyl$,
	\item[(B2)] $\mcyl \times (\partial \Delta^n)^\flat\cup \{0\} \times (\Delta^n)^\flat \to \mcyl \times (\Delta^n)^\flat$.
\end{itemize}
\end{defn}

\begin{defn}
A map $X^+\to A^+$ is called a \emph{marked left fibration} if it has the right lifting property with respect to the class of marked left anodyne extensions. A marked simplicial set $X^+$ is called \emph{marked left fibrant} if the map $X^+\to \Delta^0$ is a marked left fibration.
\end{defn}

\begin{sub}
There are dual notions of marked \emph{right} anodyne extensions and marked \emph{right} fibrations. The class of marked right anodyne extensions has generators (A1) and (A2) of Definition \ref{markedgenerators} and the classes
\begin{itemize}
	\item[(B1')] $\mcyl \times (\Delta^1)^\flat \cup \{1\} \times \mcyl \to \mcyl \times \mcyl$,
	\item[(B2')] $\mcyl \times (\partial \Delta^n)^\flat\cup \{1\} \times (\Delta^n)^\flat \to \mcyl \times (\Delta^n)^\flat$. 
\end{itemize}
The marked right fibrations have the right lifting property against the marked right anodyne extensions.
\end{sub}

\begin{thm}
Let $A^+$ be a marked simplicial set. Then there is a unique model structure on the category $\mathbf{sSet}^+/A^+$ with 
\begin{enumerate}
	\item cofibrations given by maps whose underlying map of simplicial sets is a monomorphism,
	\item fibrant objects given by marked left fibrations $X^+\to A^+$.
\end{enumerate}
Moreover, the fibrations between fibrant objects are precisely the marked left fibrations.
\end{thm}

\begin{proof}
See \cite[Theorem 4.29]{contrakim} for this precise statement. An alternative proof of the existence of this model structure when $A^+= A^\sharp$ is in \cite[Proposition 3.1.3.7]{lurie}.
\end{proof}

\begin{defn}
We will call the model structure on \(\mSet/A^+\) the \emph{coCartesian model structure}. We denote its homotopy category by \(\coCart(A^+)\). The dual model structure will be called the \emph{Cartesian model structure} and its homotopy category is denoted by \(\mathbf{Cart}(A^+)\).
\end{defn}

\begin{sub}
A useful property of coCartesian fibrations is that pullback along them preserves \emph{cellular marked right anodyne extensions}, which are those marked right anodyne extensions lying in the saturated class generated by the sets (B1) and (B2) in Definition \ref{markedgenerators}.
\end{sub}

\begin{thm}\label{ezproper}
Consider a pullback square of marked simplicial sets
\[
\begin{tikzcd}
X^+ \rar{j} \dar[swap]{p}& Y^+ \dar{q}\\
A^+\rar{i} & B^+	
\end{tikzcd}
\]
where $p$ and $q$ are marked left fibrations and $i$ is a cellular marked right anodyne extension. Then $j$ is a marked right anodyne extension. The dual statement for marked right fibrations and cellular marked left anodyne extensions also holds.
\end{thm}

\begin{proof}
See \cite[4.45]{contrakim}.	
\end{proof}

\begin{sub}
One of the main goals of this article is to extend this Theorem to general marked right anodyne extensions, see section \ref{propersection}. The main difficulty is to prove that pulling back along (co)Cartesian fibrations preserves Joyal equivalences. A proof of this fact using straightening/unstraightening can be found in \cite[Appendix B.3]{lurieha}. However, the goal of this note is to prove properties of (co)Cartesian fibrations without straightening/unstraightening and instead (eventually) derive it as a consequence. A proof for left/right fibrations without straightening/unstraightening has appeared in \cite[Proposition 5.3.5]{cisinskibook} and we will make use of this fact in the following Proposition.
\end{sub}

\begin{prop}\label{improvedrightproper}
In the Theorem above, if the underlying map of simplicial sets \(q\colon Y \to B\) is a left (resp. right) fibration and \(i\) is a marked right (resp. left) anodyne extension, then the map \(j\) is a marked right (resp. left) anodyne extension.
\end{prop}

\begin{proof}
It suffices to show this when the map \(i\colon A^+\to B^+\) belongs to the class of generators for marked right anodyne extensions.

We first show this for the class (A1). We need to show that for any diagram of pullback squares of the form
\[
\begin{tikzcd}
	Y^+\rar{j} \dar & X^+ \dar\\
	(\innerhorn)^\flat \rar{i} & (\nsimplex)^\flat 
\end{tikzcd}
\]
where the vertical maps are marked left fibrations with underlying map of simplicial sets being left fibrations and $i$ is inner horn inclusion, the map $j$ is a marked right anodyne extension. We observe that the marked simplicial sets $Y^+$ and $X^+$ have precisely the equivalences marked. In particular $X^+$ is fibrant over the point. Since the underlying map of simplicial sets
\[
Y\to X
\]
is a Joyal trivial cofibration by \cite[Proposition 5.3.5]{cisinskibook}, the map 

\[
Y^\flat\to X^\flat
\]
is a trivial cofibration over the point. We have a square
\[
\begin{tikzcd}
Y^\flat\rar \dar & X^\flat\dar\\
Y^+\rar & X^+
\end{tikzcd}
\]
Here, the vertical maps are marked left anodyne, since they are given by marking equivalences and the upper horizontal map is a trivial cofibration. Thus the lower horizontal map is a trivial cofibration over the point. Since $X^+$ is fibrant, this map is in fact marked right anodyne by \cite[2.31]{contrakim}. 

For the class (A2) it suffices to show that for a pullback square
\[
\begin{tikzcd}
Y^+\rar{i} \dar & X^\sharp\dar\\
J^\flat \rar & J^\sharp
\end{tikzcd}
\]
the map $i$ is marked left anodyne. Clearly, the map $i$ is the identity on underlying simplicial sets and $i$ is obtained by marking equivalences, thus is marked left anodyne.

The classes (B1') and (B2') follow from Theorem \ref{ezproper} since they are cellular marked right anodyne.
\end{proof}


\section{Invariance properties of the (co)Cartesian model structure}

\begin{sub}
The goal of this section is to show that for any Joyal equivalence $A\to B$ we obtain a Quillen equivalence of coCartesian (resp. Cartesian) model structures $\mSet/(A^\sharp)\to \mSet/(B^\sharp)$.
\end{sub}

\begin{thm}\label{inner-anodyne-qequivalence}
Let \(i\colon A\to B\) be inner anodyne. Then the induced functor
\[
i_!\colon \mSet(A^\sharp)\to \mSet(B^\sharp)
\]
is a Quillen equivalence for the coCartesian and Cartesian model structures.
\end{thm}

\begin{proof}
We only prove the coCartesian case as the Cartesian case is analogous. It is clear that \(i_!\) is left Quillen, thus it suffices to show that the left derived functor \(\mathbf Li_!\) is an equivalence of categories. Since \(i\) is a bijection on objects, the right derived functor \(\mathbf Ri^*\) is conservative, hence it suffices to show that \(\mathbf Li_!\) is fully faithful. 

Consider a commutative square
\[
\begin{tikzcd}
X^\natural \rar{j} \dar[swap]{p} & Y^\natural\dar{q}\\
A^\sharp \rar{i} & B^\sharp
\end{tikzcd}
\]
in which \(p\) and \(q\) are coCartesian fibrations and \(j\) is marked left anodyne. We prove that the induced map on fibers is an equivalence (in the (co)Cartesian model structure over the point). Consider a point \(a\colon \Delta^0\to A\). Choose a commutative square
\[
\begin{tikzcd}
\Delta^0\rar{u} \dar[swap] & E^\sharp\dar{f}\\
A^\sharp \rar & B^\sharp
\end{tikzcd}
\]
in which \(u\) is cellular marked right anodyne and \(f\) is a marked right fibration and denote \(F^\sharp := A^\sharp \times_{B^\sharp}E^\sharp\). Note that the induced map \(\Delta^0\to D^\sharp\) is cellular marked right anodyne by \cite[Theorem 5.2.14]{cisinskibook} and \(i\) is a marked left anodyne extension. It follows again from \cite[Theorem 5.2.14]{cisinskibook} that the pullback square
\[
\begin{tikzcd}
F^\sharp\rar\dar & E^\sharp\dar\\
A^\sharp \rar & B^\sharp
\end{tikzcd}
\]
has horizontal arrows cellular marked right fibrations and vertical arrows marked right fibrations. Pulling back the maps \(p\) and \(q\) along this pullback square, we obtain a commutative diagram of pullback squares
\[
\begin{tikzcd}
X^\natural_a \rar \dar & Y^\natural_a\dar\\
X^\natural_F\rar \dar & Y^\natural_E \dar\\
X^\natural \rar & Y^\natural
\end{tikzcd}
\]
Here we define
\[
X^\natural_F:= F^\sharp\times_{A^\sharp} X^\natural,\quad Y^\natural_E:= E^\sharp\times_{B^\sharp} Y^\natural
\]
The maps
\[
X^\natural_F\to X^\natural, \quad Y^\natural_E \to Y^\natural
\]
are marked right fibrations whose underlying map of simplicial sets are right fibrations. Since \(X^\natural \to Y^\natural\) was assumed to be marked left anodyne it follows from Proposition \ref{improvedrightproper} that
\[
X^\natural_F\to Y^\natural_E
\]
is marked left anodyne. Also by Theorem \ref{ezproper} the maps
\[
X^\natural_a \to X^\natural_F,\quad Y^\natural_a \to Y^\natural_E
\]
are marked right anodyne extensions. Since marked right anodyne and marked left anodyne extensions are in particular weak equivalences in the (co)Cartesian model structure over the point, it follows by 2-out-of-3 that the map
\[
X^\natural_a \to Y^\natural_a
\]
is an equivalence. This shows that the derived unit is an isomorphism and hence \(\mathbf Li_!\) is fully faithful.
\end{proof}

\begin{coro}\label{stability-of-inner-anodyne}
Let
\[
\begin{tikzcd}
Y^\natural \rar{j} \dar & X^\natural\dar\\
A^\sharp\rar{i} & B^\sharp
\end{tikzcd}
\]
be a pullback square of marked simplicial sets. Suppose \(i\) is inner anodyne and \(X^\natural \to B^\sharp\) is a marked left (resp. right) fibration. Then the map \(j\) is a marked left (resp. right) anodyne extension.
\end{coro}

\begin{proof}
We show this for marked left fibrations. Since marked left anodyne extensions are saturated, it suffices to show this for inner horn inclusions. By the previous Theorem \ref{inner-anodyne-qequivalence}, the functor 
\[
i^\ast \colon \mSet/(\Delta^n)^\sharp \to \mSet /(\Lambda_k^n)^\sharp
\]
is a Quillen equivalence for the coCartesian model structures. Hence the map \(j\) is a trivial cofibration over \(B^\sharp\) with fibrant target, thus by \cite[2.31]{contrakim} a marked left anodyne extension.
\end{proof}

\begin{coro}
Let \(i\colon A\to B\) be a Joyal equivalence. Then the induced functor
\[
i_!\colon \mSet/A^\sharp\to \mSet/B^\sharp
\]
is a Quillen equivalence when both categories are endowed with the (co)Cartesian model structure.
\end{coro}

\begin{proof}
Let \(W\) be the class of morphisms for which $i_!$ is a Quillen equivalence. We want to show that the class $W$ contains the Joyal equivalences. According to \cite[3.6.2]{cisinskibook} it suffices to show that 
\begin{enumerate}
	\item $W$ is closed under 2-out-of-3,
	\item $W$ contains the inner anodyne extensions
	\item $W$ contains the trivial fibrations.
\end{enumerate} 
The first assertion follows from the fact that Quillen equivalences are closed under 2-out-of-3. The second assertion follows from the previous Theorem and the third assertion is clear.
\end{proof}

\begin{Remark}
A proof along similar lines has appeared as \cite[Theorem 5.15]{short_StrUnStr}.
\end{Remark}


\section{Properness of (co)Cartesian fibrations}\label{propersection}

\begin{sub}
In this section we prove that the pullback of an inner anodyne map along a (co)Cartesian fibration is a Joyal equivalence. This has first appeared in the literature in \cite[Proposition 3.3.1.3]{lurie}, but the proof uses the straightening/unstraightening equivalence. Our proof will only use elementary properties of (locally) coCartesian fibrations. Together with Theorem \ref{ezproper}, this generalizes Proposition \ref{improvedrightproper}.
\end{sub}

\begin{sub}
Recall that a morphism of simplicial sets $i\colon A\to B$ is called \emph{final}, if for any morphism $f\colon B\to C$, the induced morphism $i\colon (A,fi)\to (B,f)$ in the slice $\mathbf{sSet}/C$ is a Contravariant equivalence. A monomorphism is final if and only if it is right anodyne, see \cite[Corollary 4.1.9]{cisinskibook}.

Furthermore recall that a morphism $p\colon X\to Y$ is called \emph{proper}, if for any diagram
\[
\begin{tikzcd}
	A'\rar{i'}\dar & B'\rar \dar & X\dar{p}\\
	A\rar{i} & B\rar & Y
\end{tikzcd}
\]
in which the squares are pullbacks and the map $i$ is final, it follows that $i'$ is final. Examples of proper morphisms are left fibrations, see \cite[Proposition 4.4.11]{cisinskibook} and coCartesian fibrations, see \cite[Corollary 4.46]{contrakim} (or \cite[Proposition 4.1.2.15]{lurie} using straightening).
\end{sub}

\begin{thm}
Let $p\colon X\to Y$ be an inner fibration of $\infty$-categories. Then $p$ is proper if and only if in any diagram
\[
\begin{tikzcd}
	A'\rar{i'}\dar & B'\rar \dar{p'} & X\dar{p}\\
	\{1\}\rar & \Delta^1\rar & Y
\end{tikzcd}
\]
in which the squares are pullbacks, the map $i'$ is final.
\end{thm}

\begin{proof}
This is \cite[4.4.36]{cisinskibook}.
\end{proof}

\begin{coro}\label{locoprop}
Locally coCartesian fibrations between $\infty$-categories are proper.
\end{coro}

\begin{proof}
In the diagram of the theorem, if $p$ is locally coCartesian then the pullback $p'$ is a coCartesian fibration. Since coCartesian fibrations are proper, the assertion follows.
\end{proof}

\begin{lem}\label{joyalleftfib}
Suppose we have a diagram
\[
\begin{tikzcd}
	X\arrow{rr}{i}\drar[swap]{p} & & Y\dlar{q}\\
	& S &
\end{tikzcd}
\]
in which $p$ is a left fibration, $i$ is a trivial cofibration of the Joyal model structure and $q$ is a Joyal fibration. Then $q$ is also a left fibration.
\end{lem}

\begin{proof}
Choose a factorization
\[
\begin{tikzcd}
	Y\arrow{rr}{j}\drar[swap]{q} & & Z\dlar{r}\\
	& S &
\end{tikzcd}
\]
where $j$ is left anodyne and $r$ is a left fibration. Since Joyal equivalences are cofinal \cite[Proposition 5.3.1]{cisinskibook}, the composition $ji$ is cofinal. Since this determines a Covariant equivalence between the left fibrations $p$ and $r$, it follows that $ji$ is in fact a Joyal equivalence. Since $i$ was assumed to be a Joyal equivalence, it follows that $j$ is also a Joyal equivalence. By the Retract Lemma, the map $q$ is a retract of $r$ and thus a left fibration.
\end{proof}

\begin{lem}\label{locolem}
Consider the commutative diagram
\[
\begin{tikzcd}
E^\natural \dar[swap]{e} & \\
X^\natural\dar[swap]{p}\rar{j} & Y^\natural\dar{q}\\
(\Lambda^n_k)^\sharp\rar{i} & (\Delta^n)^\sharp
\end{tikzcd}
\]
in which the lower square is a pullback, each vertical arrow is a marked left fibration and the map of simplicial sets $i$ is inner anodyne. Assume furthermore that the underlying map of $e$ is a left fibration of simplicial sets. Choose a factorization
\[
\begin{tikzcd}
E^\natural \dar[swap]{e}\rar{k} & F^\natural\dar{f}\\
X^\natural\rar{j} & Y^\natural
\end{tikzcd}
\]
where $k$ is marked left anodyne and $f$ is a marked left fibration. Then the underlying map of $f$ is a locally coCartesian fibration. Moreover, for each 0-simplex $x\colon \Delta^0 \to X$, the induced map on fibers $E_x\to F_x$ is a Joyal equivalence.
\end{lem}

\begin{proof}
Note that the markings on $E^\natural$ correspond to the $pe$-coCartesian edges and that markings on $F^\natural$ correspond to the $qf$-coCartesian edges. First we observe that the base-change map
\[
E^\natural \to \left(\Lambda^n_i\right)^\sharp \times_{\left(\Delta^n\right)^\sharp} F^\natural
\]
is marked left anodyne by Corollary \ref{inner-anodyne-qequivalence}. Consequently, for each object $m\in \Delta^n$, the induced map on fibers
\[
E_m \to F_m
\]
is a Joyal equivalence.

Next we observe that $f$ is an isofibration, since the marked edges in $Y^\natural$ are precisely the $q$-coCartesian edges and thus in particular the equivalences are marked. We obtain a commutative diagram
\[
\begin{tikzcd}
E_m \arrow{rr}{k_m}\drar[swap]{e_m} & & F_m\dlar{f_m}\\
& X_m & 
\end{tikzcd}
\]
where $k_m$ is a trivial cofibration of the Joyal model structure, $e_m$ is a right fibration (by assumption) and $f_m$ is an isofibration between $\infty$-categories, thus a Joyal fibration. By Lemma \ref{joyalleftfib} the map $f_m$ is thus a left fibration. In other words, the map $F\to X$ induces for each $m\in \Delta^n$ a left fibration on fibers $F_m\to X_m$. Now since $k_m$ is a Joyal equivalence, it is in particular a Covariant equivalence between the left fibrations $e_m$ and $f_m$ and thus a fiberwise equivalence.

It remains to show that $f\colon F\to Y$ is locally coCartesian. By construction, we have a commutative diagram
\[
\begin{tikzcd}
F\arrow{rr}{f}\drar[swap]{qf}& & Y\dlar{q}\\
& \Delta^n &	
\end{tikzcd}
\]
with $qf$ and $q$ being coCartesian fibrations and $f$ sending $qf$-coCartesian edges to $q$-coCartesian edges. Thus by \cite[2.4.2.11]{lurie} the map $f$ is locally coCartesian since the maps $f_m\colon F_m\to Y_m$ are in fact left fibrations.
\end{proof}

\begin{thm}\label{joyalproperness}
Suppose we have a pullback square
\[
\begin{tikzcd}
X\dar[swap]{p}\rar{j} & Y\dar{q}\\
A\rar{i} & B
\end{tikzcd}
\]
in which $p$ and $q$ are coCartesian fibrations and $i$ is inner anodyne. Then $j$ is a Joyal equivalence.
\end{thm}

\begin{proof}
It suffices to show the assertion for squares of the form
\[
\begin{tikzcd}
X\dar[swap]{p}\rar{j} & Y\dar{q}\\
\Lambda^n_k\rar{i} & \Delta^n
\end{tikzcd}
\]
with $i$ an inner horn inclusion.

According to \cite[5.2.16]{cisinskibook} the map $j$ is a Joyal equivalence if and only if it induces an essentially surjective functor on homotopy categories and it induces a fully faithful functor
\[
\mathbf{L}j_!\colon \mathbf{LFib}(X)\to \mathbf{LFib}(Y)
\]
where \(\mathbf{LFib}(X)\) denotes the homotopy category of the covariant model structure on simplicial sets over \(X\). Since $j$ is a pullback of an inner anodyne map, it is a bijection on objects, thus clearly essentially surjective. For the second condition, we need to show that the derived counit
\[
id \to \mathbf{R}j^\ast \mathbf{L}j_!
\]
is an isomorphism in $\mathbf{LFib}(X)$. Let $e\colon E\to X$ be a left fibration. To prove that the derived counit is an isomorphism, we construct a particular fibrant replacement of the composition $E\to X\to Y$ in the covariant model structure over \(Y\). 

To start off, we have a pullback square in marked simplicial sets
\[
\begin{tikzcd}
X^\natural\dar[swap]{p}\rar{j} & Y^\natural\dar{q}\\
(\Lambda^n_k)^\sharp\rar{i} & (\Delta^n)^\sharp	
\end{tikzcd}
\]
in which $p$ and $q$ are marked left fibrations. Since $e$ is a left fibration, we have a marked left fibration $e\colon E^\sharp \to X^\sharp$ and pulling back along the inclusion $X^\natural \to X^\sharp$, we obtain a marked left fibration $E^\natural \to X^\natural$. We thus have a diagram
\[
\begin{tikzcd}
E^\natural\dar[swap]{e} & \\
X^\natural\dar[swap]{p}\rar{j} & Y^\natural\dar{q}\\
(\Lambda^n_k)^\sharp\rar{i} & (\Delta^n)^\sharp
\end{tikzcd}
\]
where the markings on $E^\natural$ correspond to the $pe$-coCartesian edges.

Now complete the diagram as follows,
\[
\begin{tikzcd}
E^\natural \dar[swap]{e}\rar{k} & F^\natural\dar{f}\\
X^\natural\dar[swap]{p}\rar{j} & Y^\natural\dar{q}\\
(\Lambda^n_k)^\sharp\rar{i} & (\Delta^n)^\sharp
\end{tikzcd}
\]
with $k$ marked left anodyne and $f$ a marked left fibration. The markings on $F^\natural$ thus correspond to the $qf$-coCartesian edges. The underlying map of simplicial sets $f\colon F\to Y$ is in general not a coCartesian fibration, but the previous Lemma \ref{locolem} shows that it is a locally coCartesian fibration and that for any $x\in X$ the induced map $E_x\to F_x$ is a Joyal equivalence (between Kan complexes).

Now find a factorization
\[
\begin{tikzcd}
F\arrow{rr}{l}\drar[swap]{f} & & G\dlar{g}\\
& Y &
\end{tikzcd}
\]
with $l$ left anodyne and $g$ a left fibration. By Lemma \ref{locoprop} the map $f$ is proper and thus by \cite[Corollary 4.4.28]{cisinskibook} the induced map on fibers $F_y\to G_y$ is  cofinal for any $y\in Y$.

To summarize, we have constructed a commutative diagram 
\[
\begin{tikzcd}
E\rar{k}\dar[swap]{e} & F\rar{l}\dar[swap]{f} & G\dlar{g}\\
X\rar{j} & Y	
\end{tikzcd}
\]
in which $k$ and $l$ are left anodyne and $g$ is a left fibration, thus $g$ is a fibrant replacement for the composition $je$. Now for any $x\in X$ we have induced maps on fibers
\[
E_x \xrightarrow{k_x} F_x \xrightarrow{l_x} G_x
\]
where $k_x$ is cofinal by Lemma \ref{locolem} and $l_x$ is cofinal by the above arguments. In particular this implies that the induced map
\[
E\to X\times_Y G
\]
is cofinal which in turn implies that the derived counit is an isomorphism. Thus $X\to Y$ is fully faithful and this finishes the proof.
\end{proof}

\begin{sub}
As a consequence we show that the coCartesian model structure is functorial with respect to coCartesian fibrations.	
\end{sub}

\begin{thm}\label{pbleftquillen}
Let $p\colon X^\natural \to A^\sharp$ be a Cartesian fibration (i.e. a marked right fibration). Then the pullback functor
\[
p^\ast \colon \mSet/A^\sharp\to \mSet/X^\natural
\]
is a left Quillen functor when each category is endowed with the coCartesian model structure.
\end{thm}

\begin{proof}
Since marked simplicial sets are locally cartesian closed, the functor $p^\ast$ is a left adjoint. To show that $p^\ast$ is left Quillen, it suffices to show that it preserves marked left anodyne extensions. In particular, it suffices to show that $p^\ast$ preserves the generating marked left anodyne extensions of Definition \ref{markedgenerators}.

We first show this for set (A1). We need to show that for any diagram of pullback squares of the form
\[
\begin{tikzcd}
	X^\natural_{\innerhorn}\rar{j} \dar & X^\natural_{\nsimplex} \dar \rar & X^\natural \dar{p}\\
	(\innerhorn)^\flat \rar{i} & (\nsimplex)^\flat \rar & A^\sharp,
\end{tikzcd}
\]
where $i$ is inner horn inclusion, the map $j$ is a trivial cofibration in $\mSet/X^\natural$. We observe that the marked simplicial sets $X^\natural_{\innerhorn}$ and $X^\natural_{\nsimplex}$ have precisely the equivalences in their fibers over $(\innerhorn)^\flat$ and $(\nsimplex)^\flat$ marked. In particular $X^\natural_{\Delta^n}$ is fibrant over the point. Since the underlying map of simplicial sets
\[
X_{\Lambda^n_k}\to X_{\Delta^n}
\]
is a Joyal trivial cofibration by Theorem \ref{joyalproperness}, the map 

\[
X^\flat_{\Lambda^n_k}\to X^\flat_{\Delta^n}
\]
is a trivial cofibration over the point. We have a square
\[
\begin{tikzcd}
X_{\Lambda^n_k}^\flat\rar \dar & X_{\Delta^n}^\flat\dar\\
X_{\Lambda^n_k}^\natural\rar & X^\natural_{\Delta^n}
\end{tikzcd}
\]
Here, the vertical maps are marked left anodyne, since they are given by marking equivalences and the upper horizontal map is a trivial cofibration. Thus the lower horizontal map is a trivial cofibration over the point. Since $X^\natural_{\Delta^n}$ is fibrant, this map is in fact marked left anodyne by \cite[2.31]{contrakim} and thus a trivial cofibration in $\mSet/X^\natural$. 
\end{proof}

\section{Straightening and unstraightening}

\input{betterrectification}

\bibliographystyle{alpha}
\bibliography{directreferences}

\end{document}

%% file: betterrectification.tex

\begin{sub}
This section proves a straightening/unstraightening equivalence for coCartesian fibrations. For this we assume that \(A\) is the nerve of a small category. We will prove the following Theorem.
\end{sub}

\begin{thm}\label{rectificationeq}
For any simplicial set \(B\) there is a Quillen equivalence
\[
\Fun(A,\mSet/B^\sharp) \simeq \mSet/A^\sharp \times B^\sharp
\]
where the right hand side is endowed with the Cartesian model structure and the left hand side is endowed with the projective Cartesian model structure.
\end{thm}

\begin{sub}
Contrary to the existing literature \cite{lurie}, \cite{short_StrUnStr}, our proof does not involve simplicial categories. Our equivalence is also \emph{not} the full straightening/unstraightening equivalence as our base category is assumed to be the nerve of a category. Our proof follows ideas from \cite{heutsmoerdijk}.
\end{sub}

\begin{sub}
We first define the functors involved. Let \(X^+\to A^\sharp\) be a map and \(B^+\) be a marked simplicial set. We have a functor
\[
\mSet/B^+\to \mSet/A^\sharp \times B^+
\]
given by sending a map \(Y^+\to B^+\) to the product \(X^+\times Y^+\to A^\sharp \times B^+\). This functor has a right adjoint
\[
\Map^B(X^+,-) : \mSet/A^\sharp \times B^+\to \mSet/B^+
\]
By the universal property a map
\[
\begin{tikzcd}
K^+\ar{rr}\drar & & \Map^B(X^+,W^+)\dlar\\
& B^+ &
\end{tikzcd}
\]
is thus given by a commutative triangle
\[
\begin{tikzcd}
X^+ \times K^+\ar{rr}\drar & & W^+\dlar\\
& A^\sharp \times B^+ &
\end{tikzcd}
\]
Note that \(\Map^B(X^+,W^+)\) is (contravariantly) functorial in \(X^+\to A^+\). This defines a functor
\[
\rho : \mSet/A^\sharp \times B \to \mathbf{Fun}(A,\mSet/B^+),\quad W^+ \mapsto (a\mapsto \Map^B(\coslice,W^+)
\]
Given a map \(p : W^+\to A^\sharp \times B^+\) The marked simplicial set \(\Map^{B^+}(X^+,W^+)\) can be described as the pullback
\[
\begin{tikzcd}
\Map^{B^+}(X^+,W^+)\rar \dar & \Hom^+(X^+,W^+)\dar{p_\ast}\\
B^+\rar & \Hom^+(X^+,A^\sharp \times B^+ )
\end{tikzcd}
\]
where the bottom map is given by the product of the fixed map \(X^+\to A^\sharp\) and the identity on \(B^+\).
\end{sub}

\begin{sub}
This functor has a left adjoint
\[
\lambda : \mathbf{Fun}(A,\mSet/B^+)\to \mSet/A^\sharp \times B^+.
\]
Given a functor \(F : A\to \mSet/B^+\) we obtain the functor
\[
A^{op}\times A \to \mSet/A^\sharp \times B^+,\quad (a,a)\mapsto (\coslice\times F(a)^+\to A^\sharp \times B^+)
\]
The value of the left adjoint is then given by taking the coend
\[
\lambda(F)= \int^{A}\coslice \times F(a)^+
\]
\end{sub}

\begin{Example}\label{generators}
Let \(X^+\to B^+\) be a map of marked simplicial sets and let \(a\) be an object of \(A\). Let \(a\otimes X^+\) be the functor given by left Kan extension along the inclusion \(\{a\}\to A\). Then we have
\[
\lambda(a\otimes X^+)=
\begin{tikzcd}
\coslice\times X^+ \dar\\
A^\sharp \times B^+
\end{tikzcd}
\]
\end{Example}

\begin{prop}
For any coCartesian equivalence \(K^+\to L^+\) over \(A^\sharp\) and any marked left fibration \(X^+\) over \(A^\sharp \times B^+\) the induced map
\[
\Map^{B^+}(L^+,X^+) \to \Map^{B^+}(K^+,X^+)
\]
is a coCartesian equivalence over \(B^+\).
\end{prop}

\begin{proof}
We have a pullback square
\[
\begin{tikzcd}
\Map^{B^+}(L^+,X^+)\rar \dar & \underline \Hom^+(L^+,X^+)\dar\\
\Map^{B^+}(K^+,X^+)\rar & \underline \Hom^+ (K^+,X^+)\times_{\underline \Hom^+(K^+, A^\sharp\times B^+)} \underline \Hom^+(L^+,A^\sharp \times B^+)
\end{tikzcd}
\]
Since the right hand vertical map is a trivial fibration whenever \(K^+\to L^+\) is marked left anodyne, the functor \(\Map_{A^\sharp}^{B^+}(-,X^+)\) sends trivial cofibrations to weak equivalences in the opposite of the cocartesian model structure on \(\mathbf{sSet^+}/B^+\). 
\end{proof}

\begin{coro}\label{fibercomputation}
Suppose \(X^+\to A^\sharp \times B^+\) is a marked left fibration. Let \(X^+_a\) be the pullback
\[
\begin{tikzcd}
X^+_a\rar \dar & X^+\dar\\
\{a\}\times B^+\rar & A^\sharp \times B^+
\end{tikzcd}
\]
Then there is a coCartesian equivalence
\[
X^+_a \simeq \Map_{A^\sharp}^{B^+}(\coslice ,X^+)
\]
over \(B^+\).
\end{coro}

\begin{proof}
The map \(\{a\}\to \coslice\) is marked left anodyne.
\end{proof}

\begin{prop}
The right adjoint preserves fibrations between fibrant objects.
\end{prop}

\begin{proof}
Let 
\[
\begin{tikzcd}
X^+\drar \ar{rr} & & Y^+\dlar\\
& A^\sharp \times B^+ &
\end{tikzcd}
\]
be a fibration between marked left fibrations. In particular by \cite{contrakim} it is a marked left fibration. We have a pullback square
\[
\begin{tikzcd}
\Map_{A^\sharp}^{B^+}(K^+,X^+)\rar \dar & \underline \Hom^+(K^+,X^+)\dar\\
\Map_{A^\sharp}^{B^+}(K^+,Y^+)\rar & \underline \Hom(K^+,Y^+)
\end{tikzcd}
\]
The right hand side is a marked left fibration hence the pullback is a marked left fibration between marked left fibrations over \(B^+\).
\end{proof}

\begin{coro}
The functors determine a Quillen adjunction
\[
\lambda \colon \Fun (A,\mathbf{sSet}^+/B^+) \leftrightarrow \mathbf{sSet}^+/A^\sharp \times B^+ \colon \rho
\]
\end{coro}

\begin{sub}
In order to prove that this defines a Quillen equivalence, we show that this Quillen adjunction respects evaluation at a point in \(A\). It is easy to see that the following square commutes for each object 
\(a\) of \(A\):
\begin{equation}\label{evaluation}
\begin{tikzcd}
\Fun(A,\mSet/B^+)\rar{\lambda} \dar[swap]{ev_a} & \mSet/A^\sharp \times B^+\dar{a^*}\\
\mSet/B^+ \rar{id} & \mSet/B^+ 
\end{tikzcd}
\end{equation}
\end{sub}

\begin{prop}\label{mate}
The induced transformation
\[
ev_a\rho \Rightarrow a^*
\]
is a coCartesian equivalence for each fibrant object of \(\mSet/A^\sharp \times B^+\).
\end{prop}

\begin{proof}
Let \(W^+\to A^\sharp \times B^+\) be a coCartesian fibration. We need to show that the counit
\[
\lambda \rho (W^+)\to W^+
\]
induces a coCartesian equivalence after taking fibers at the object \(a\) of \(A\). The map
\[
\lambda \rho (W^+)_a\to W^+_a
\]
can be written as
\[
\Map^{B^+}(\coslice,W^+) \to W_a^+
\]
and an explicit computation shows that this coincides with the map of Corollary \ref{fibercomputation}, hence is a trivial fibration.
\end{proof}

\begin{sub}
One of the key observations is that the pullback functor preserves homotopy colimits, see also \cite[Lemma 3.5]{short_StrUnStr}.
\end{sub}

\begin{prop}
Let \(f\colon K\to L\) be a map of simplicial sets. Then the induced functor
\[
f^\ast\colon \mSet/L^\sharp \to \mSet/K^\sharp
\]
preserves homotopy colimits.
\end{prop}

\begin{proof}
We won't give a full proof here, as we cannot improve on the proof of \cite{short_StrUnStr} or offer a different viewpoint. The idea is to reduce to the case when \(K\cong \Delta^0\) and thus \(f\colon \Delta^0 \to L\) is the specification of an object in \(L\). Then one observes that the pullback \(f^\ast\) is weakly equivalent to pulling back along a marked left fibration. By Theorem \ref{pbleftquillen} pulling back along a marked left fibration is left Quillen, thus it preserves homotopy colimits and the assertion follows.
\end{proof}

\begin{coro}\label{hocolimfibers}
Let \(a\) be an object of \(A\). Then the functor 
\[
\mathbf{R}a^\ast : \mSet/A^\sharp \times B^+\to \mSet/B^+
\]
preserves homotopy colimits.
\end{coro}

\begin{coro}
The functor \(ev_a\) preserves homotopy colimits.
\end{coro}

\begin{proof}
This follows immediately from the commutative diagram (\ref{evaluation}) and the previous Corollary.
\end{proof}

\begin{coro}\label{righthocolim}
The functor \(\rho\) preserves homotopy colimits.
\end{coro}

\begin{proof}
We show that the map
\[
hocolim_I \mathbf{R}\rho \to \mathbf{R}\rho\ hocolim_I
\]
is an equivalence. Since equivalences are computed point wise, it is enough to show that
\[
\mathbf{R}ev_a hocolim_I \mathbf{R}\rho \to \mathbf{R}ev_a \mathbf{R}\rho\ hocolim_I
\]
is an equivalence. By the previous corollary \(ev_a\) commutes with homotopy colimits and by Proposition \ref{mate} we have
\[
\mathbf{R}ev_a \mathbf{R}\rho \simeq \mathbf{R}a^\ast
\]
Thus the assertion follows from Proposition \ref{hocolimfibers}.
\end{proof}

\begin{lem}
Let \(X^\natural\to A^\sharp\) be a cartesian fibration and \(Y^+\to B^\sharp\) be a map. Let \(Z^\natural\to A^\sharp \times B^\sharp\) be a cartesian fibration. Then a map
\[
\begin{tikzcd}
X^\natural \times Y^+\ar{rr}{f}\drar & & Z^\natural \dlar\\
& A^\sharp \times B^\sharp &
\end{tikzcd}
\]
is a cartesian equivalence if and only if for all points \(a\) of \(A\) the map
\[
\begin{tikzcd}
X^\natural_a \times Y^+\ar{rr}{f_a}\drar & & Z^\natural_a \dlar\\
& B^\sharp &
\end{tikzcd}
\]
is a cartesian equivalence.
\end{lem}

\begin{proof}
Choose a factorization
\[
Y^+\to W^\natural \to B^\sharp
\]
into a marked right anodyne extension followed by a cartesian fibration. We get an induced factorization
\[
X^\natural \times Y^+\to X^\natural \times W^\natural\to A^\sharp \times B^\sharp
\]
into marked right anodyne followed by cartesian fibration. We find a solution to the lifting problem
\[
\begin{tikzcd}
X^\natural \times Y^+\rar{f} \dar & Z^\natural \dar\\
X^\natural \times W^\natural \rar\urar[dashed]{g} & A^\sharp \times B^\sharp
\end{tikzcd}
\]
since the left hand side is marked right anodyne and the right hand side is a cartesian fibration by assumption. Thus, the map \(f\) is a cartesian equivalence if and only if the map \(g\) is. Now let us take fibers at the inclusion
\[
\{a\}\times B^\sharp \to A^\sharp \times B^\sharp. 
\]
We obtain
\[
\begin{tikzcd}
X^\natural_a \times Y^+\rar{f_a} \dar & Z^\natural_a \dar\\
X^\natural_a \times W^\natural \rar\urar{g_a} & B^\sharp
\end{tikzcd}
\]
Again, the left hand side is marked right anodyne hence \(f_a\) is a cartesian equivalence if and only if \(g_a\) is a cartesian equivalence. It thus suffices to show that \(g\) is a cartesian equivalence if and only if \(g_a\) is. This now follows from the fact that cartesian equivalences between cartesian fibrations are detected pointwise.
\end{proof}

\begin{prop}\label{derivedunit}
A map \(\lambda(a\otimes X)\to W^+\) in \(\mSet/A^\sharp \times B^+\) is a weak equivalence if and only if the adjoint map \(a\otimes X\to \rho W^+\) is.
\end{prop}

\begin{proof}
It suffices to show that for an object \(a\) of \(A\), a map \(X^+\to B^\sharp\) and a cartesian fibration \(Z^\natural\to A^\sharp \times B^\sharp\), a map 
\[
\begin{tikzcd}
\lambda(a\otimes p) \ar{rr}\drar & &  Z^\natural\dlar\\
& A^\sharp \times B^\sharp &
\end{tikzcd}
\]
is a cartesian equivalence if and only if the adjoint map
\[
\begin{tikzcd}
a\otimes p \ar{rr}& & \rho(Z^\natural)
\end{tikzcd}
\]
is a pointwise cartesian equivalence. In the first case we have the map
\[
\begin{tikzcd}
A/a^\sharp \times X^+\drar \ar{rr} & & Z^\natural\dlar\\
 & A^\sharp\times B^\sharp &
\end{tikzcd}
\]
which is a cartesian equivalence if and only if for any point \(a'\) of \(A\) the induced map
\[
\begin{tikzcd}
\Hom(a',a) \times X^+\drar \ar{rr} & & Z_{a'}^\natural\dlar\\
 & B^\sharp &
\end{tikzcd}
\]
is a cartesian equivalence. In the second case, we have a pointwise cartesian equivalence if and only if for any point \(a'\) of \(A\) the induced map
\[
\begin{tikzcd}
\Hom(a',a)\times X^+\ar{rr}\drar & & \Map_A(A/a'^\sharp,Z^\natural)\dlar\\
& B^\sharp &
\end{tikzcd}
\]
is a cartesian equivalence. Since we have a cartesian equivalence \[
\Map_A(A/a'^\sharp,Z^\natural)\simeq  Z^\natural_{a'}
\]
over \(B^\sharp\), this is equivalent to the first case.
\end{proof}

\begin{proof}[Proof of Theorem \ref{rectificationeq}]
Since weak equivalences between Cartesian fibrations are computed fiberwise, it follows from Proposition \ref{mate} that \(\rho\) reflects weak equivalences between fibrant objects. It thus suffices to show that the derived unit is a weak equivalence. It follows from Proposition \ref{derivedunit} that the derived unit is a weak equivalence for each object of the form \(a\otimes X\). Since any functor is a homotopy colimit of objects of this form and the functor \(\rho\) preserves homotopy colimits by Corollary \ref{righthocolim}, it follows that the derived unit is in fact a weak equivalence.
\end{proof}